\documentclass[runningheads]{llncs}
\usepackage[T1]{fontenc}
\usepackage{graphicx, amsmath, amsfonts}
\newcommand{\id}{\text{id}}
\newcommand{\ini}{\text{in}}

\newcommand{\rank}{\textrm{Rank}}

\newcommand{\Gr}{{\rm Gr}}

\newcommand{\init}{{\rm in}}

\def\frk{\mathfrak}               

\def\a{{\bf a}}

\def\Phi{{\frk n}}
\def\Phi{{\frk N}}

\def\Ma{M_{\bf a}}

\def\wa{{w}_{\mathbf a}}

\def\opn#1#2{\def#1{\operatorname{#2}}} 
\opn\chara{char} \opn\length{\ell} \opn\pd{pd} \opn\rk{rk}
\opn\projdim{proj\,dim} \opn\injdim{inj\,dim} \opn\rank{rank}
\opn\depth{depth} \opn\grade{grade} \opn\height{height}
\opn\embdim{emb\,dim} \opn\codim{codim}

\opn\Tr{Tr} \opn\bigrank{big\,rank}
\opn\superheight{superheight}\opn\lcm{lcm}
\opn\trdeg{tr\,deg}
\opn\reg{reg} \opn\lreg{lreg} \opn\ini{in} \opn\lpd{lpd}
\opn\size{size} \opn\sdepth{sdepth}
\opn\link{link}\opn\fdepth{fdepth}\opn\lex{lex}
\opn\LM{LM}
\opn\LC{LC}
\opn\NF{NF}
\opn\Merge{Merge}
\opn\sgn{sgn}

\opn\div{div} \opn\Div{Div} \opn\cl{cl} \opn\Pic{Pic}
\opn\Prin{Prin}
\opn\op{op}
\opn\indeg{indeg} \opn\outdeg{outdeg}
\opn\red{red}
\opn\Spec{Spec} \opn\Supp{Supp} \opn\supp{supp} \opn\Sing{Sing}
\opn\Ass{Ass} \opn\Min{Min}\opn\Mon{Mon} \opn\val{val}
\opn\Ann{Ann} \opn\Rad{Rad} \opn\Soc{Soc}
 \opn\Ker{Ker} \opn\Coker{Coker} \opn\Am{Am}
\opn\Hom{Hom} \opn\Tor{Tor} \opn\Ext{Ext} \opn\End{End}
\opn\Aut{Aut} \opn\id{id}

\opn\nat{nat}
\opn\pff{pf}
\opn\Pf{Pf} \opn\GL{GL} \opn\SL{SL} \opn\mod{mod} \opn\ord{ord}
\opn\Gin{Gin} \opn\Hilb{Hilb}\opn\sort{sort}
\opn\Image{Image}
\opn\vol{Vol}
\opn\aff{aff} \opn\con{conv} \opn\relint{relint} \opn\st{st}
\opn\lk{lk} \opn\cn{cn} \opn\core{core} \opn\vol{vol}
\opn\link{link} \opn\star{star}\opn\lex{lex}\opn\set{set}
\opn\dist{dist}
\opn\gr{gr}

\begin{document}
\title{Gröbner Degenerations of Determinantal Ideals with an Application to Toric Degenerations of Grassmannians}
\titlerunning{Gröbner Degenerations of Determinantal Ideals}
%
\author{Fatemeh Mohammadi \inst{1}\orcidID{0000-0001-5187-0995}
\authorrunning{F. Mohammadi}
%
\institute{Department of Mathematics and Computer Science, KU Leuven, Belgium
\email{fatemeh.mohammadi@kuleuven.be}}
}
\maketitle              
\begin{abstract}
The concept of the Gröbner fan for a polynomial ideal, introduced by Mora and Robbiano in 1988, provides a robust polyhedral framework where maximal cones correspond to the reduced Gröbner bases of the ideal. Within this geometric structure resides the tropical variety, a subcomplex of the Gröbner fan, utilized across various mathematical domains. Despite its significance, the computational complexity associated with tropical varieties often limits practical computations to smaller instances.
In this note, we revisit a family of monomial ideals, called matching field ideals, from the context of Gr\"obner degenerations. We show that they can be obtained as weighted initial ideals of determinantal ideals. We explore the algebraic properties of these ideals, 
with a particular emphasis on minimal free resolutions and Betti numbers. 
\end{abstract}

\section{Introduction}
A Gr\"obner degeneration refers to a process in algebraic geometry where a given ideal or variety is systematically deformed to more well-behaved ideals such as monomial or toric ideals. It produces a parametric family of ideals, preserving the geometric and algebraic properties of the original ideal.
Overall, Gr\"obner degenerations provide a powerful tool for understanding the invariants of ideals and varieties under perturbations or deformations.


Monomial degenerations play a crucial role in simplifying the structure and analysis of ideals such as their minimal free resolutions~\cite{dochtermann2012cellular,HT}. For example, the semi-continuity theorem provides a bound on the Betti numbers of any ideal in terms of those of their monomial initial ideals. In this note, we will study a family of monomial degenerations of determinantal ideals~\cite{bernstein1993combinatorics,conca2015universal,sturmfels1993maximal}. 


Toric varieties are another family of popular objects in algebraic geometry; see e.g.~\cite{Cox2,Ful93}. This is mainly because there is a dictionary between their geometric properties and the combinatorial features of their polytopes. This dictionary can be extended from toric varieties to arbitrary varieties through toric degenerations. The connection between toric varieties and Newton polytopes was first developed by Khovanskii in \cite{Hov78}. 

A toric degeneration \cite{anderson2013okounkov} of an algebraic variety $X$ is a flat family over the affine line $\mathbb{A}^1$, where the fiber over zero is a toric variety and all the other fibers are isomorphic to $X$. 
Toric degenerations serve as useful tools for analyzing algebraic varieties, providing a means to comprehend a general variety through the geometry of its associated toric counterparts, as many geometric invariants (e.g. degree, dimension) remain invariant under such transformations.
Hence, computing toric degenerations of general varieties allows us to use the robust machinery of toric varieties to understand the invariants of the original varieties.
The primary focus when computing toric degenerations revolve around two pivotal questions: (1) How do we construct them? and (2) How are distinct toric degenerations of the same variety interconnected? Given the paramount importance of these questions, a plethora of methodologies have emerged across diverse fields such as algebraic geometry (e.g. \cite{anderson2013okounkov,EH-NObodies,GHKK}), representation theory \cite{FFL11Symplectic}, cluster algebras \cite{RW17,bossinger2021families}, tropical geometry \cite{bossinger2017computing,clarke-schubert,clarke2021standard,kaveh2019khovanskii,KristinFatemeh}, and combinatorics \cite{clark-Higashitani-Mohammadi,clarke2024combinatorial,deformationsflag}.
However, despite the progress made in these domains, the challenge remains: there are currently no known algorithms tailored for computing toric degenerations of a given variety or for facilitating comparisons between disparate cases from different theoretical frameworks.

\smallskip

In this note, we study a specific family of Gr\"obner degenerations, focusing on determinantal ideals and Grassmannian varieties Gr$(3,n)$. We demonstrate that such degenerations are intricately tied to weight vectors and manifest either as monomial initial ideals or maximal cones within the  Gr\"obner  fan of the determinantal ideal. Moreover, we introduce a monomial map, where the kernel materializes as a binomial prime ideal, thereby signifying a toric degeneration of the Pl\"ucker ideal of Gr$(3,n)$.
More specifically, we identify  matching field ideals as weighted monomial initial ideals of determinantal ideals, showcasing their linear quotient property. Further, we explicitly compute their Betti numbers and establish that their minimal free resolution is supported on a CW-complex, or equivalently, demonstrating their cellular resolution.


\section{Gr\"obner degeneration of determinantal ideals} 
Let $X=(x_{ij})$ represent a generic $d\times n$ matrix, and let $I_X$ denote the ideal generated by the maximal minors of $X$. For simplicity, we focus on the case where $d=3$ and we denote the variables corresponding to the first, second, and third rows of $X$ as $x$, $y$, and $z$, respectively.

\medskip

We recall the notion of block diagonal matching fields from \cite{KristinFatemeh}. See also \cite{sturmfels1993maximal,bernstein1993combinatorics}.
Consider a sequence $a=(a_1, a_2, \dots, a_r)$ of positive numbers $a_1, a_2, \dots, a_r$ such that $\sum_{i = 1}^r a_i = n$. For $1 \leq t \leq r$, define $I_t = \{\alpha_{t-1}+1, \alpha_{t-1} +2, \dots, \alpha_{t}\}$, where $\alpha_{t} = \sum_{i = 1}^{t} a_{i}$ and $\alpha_0 = 0$. We construct the ideal $M_{\bf{a}}$ as follows: For every $3$-subset $\{i<j<k\}$ of $[n]$, we let $s$ be the minimal $t$ such that $I_t \cap \{i,j,k\} \neq \emptyset$. Then we associate the monomial
\[
m_{\{ijk\}} =
\left\{
	\begin{array}{ll}
		x_jy_iz_k & \mbox{if } |\{i,j,k\}\cap I_s|=1 \\
		x_iy_jz_k & \mbox{otherwise}
	\end{array}
\right.
\]
The ideal $M_{\bf{a}}$ is generated by the monomials associated to all $3$-subsets of $[n]$. 

\begin{definition}\label{weight}
    To every $r$-block diagonal matching field of type $(a_1,\ldots,a_r)$, we associate a weight order $\prec_{\wa}$ as follows.
Let us denote $\wa(x_i), \wa(y_j),\wa(z_k)$ for the weights associated to $x_i,y_j,z_k$ for all $i,j,k$. Then 
we associate the following weights to the variables:
\begin{itemize}
    \item $
\wa(x_i) = w_0  \quad \text{for all $i$}$
\item $
\wa(y_j) = \begin{cases}
\wa(y_{\alpha_{s-1}+ 1}) + j - (\alpha_{s-1} + 1), \quad &\text{if $j \in I_s$ and $j > \alpha_{s-1}+1$ }\\
\wa(y_{\alpha_{s+1}})+1,  \quad &\text{if $j \in I_s$, $j=\alpha_{s-1} + 1$ and $s < r$}\\ 
w_0 + 1  \quad &\text{if $j=\alpha_{r-1} + 1$.}
\end{cases}
$
\item
$\wa(z_1)= \wa(z_2) = w_0\quad\text{and}\ \wa(z_i)= \wa(y_{\alpha_1}) + (n-2)(i-2) \quad \text{for all  $3 \leq i\leq n$.}
$
\end{itemize}
We let $\prec_{\text w}$ be the monomial order such that for two monomials $m,m'$ we have that $m<m'$ if the associated weight of $m$ is less than that of $m'$ or they have the same weight and $m<m'$ with respect to the revlex order induced by: 
\[\footnotesize{
z_n> 
\cdots>z_3 >\underbrace{y_{\alpha_1}>\cdots>y_{\alpha_0 +1 }}_\text{coming from $I_1$}
>\cdots>\underbrace{y_{\alpha_r}>\cdots>y_{\alpha_{r-1}+ 1}} _\text{coming from $I_r$}>z_2>z_1>x_1> \cdots>x_n.
}\]
\end{definition}

\begin{example}
Let $n = 7, {\bf a}=[12|345|67], w_0 = 1$. The order on the variables is 
\[
z_7>z_6>\cdots> z_3>y_2>y_1>y_5>y_4>y_3>y_7>y_6>z_2>z_1>x_1>\cdots>x_7
\]
The weight on the variables (represented in a matrix) 
is given as:
\[
w_{\bf a}=\begin{bmatrix}
    1 & 1 & 1 &1 &1 &1 &1\\
    7 & 8 & 4 &5 &6 &2 &3 \\
    1 & 1 & 13 &18 &23 &28 &33 \\
\end{bmatrix}.
\]
This induces $\prec_{\wa}$ weight order and $M_{\bf a} = \init_{\prec_{\wa}}(I)$.
\end{example}

\begin{theorem}\label{thm:main}
$M_{\a}$ is the initial ideal of the determinantal ideal $I_X$ w.r.t. $\prec_{\wa}$. In particular, the maximal minors form a Gr\"obner basis for $I_X$ w.r.t. 
$\prec_{\wa}$.
\end{theorem}

\begin{proof}
We first show that 
$M_\a \subseteq \init_{\prec_{\wa}}(I_X)$.
For every $3$-subset $\{i<j<k\}$ of $[n]$, let $s$ be the minimal $t$ such that $I_t \cap \{i,j,k\} \neq \emptyset$. Consider the minor  $[ijk]$ of $X$ on the columns $i,j,k$. Note that:
\[
[ijk] = \sum_{\sigma \in {\rm Sym}(i,j,k)} {\rm sign}(\sigma)x_{\sigma(i)}y_{\sigma(j)}z_{\sigma(k)} \ \text{and}\  \wa(x_iy_jz_k) = \wa(x_i) + \wa(y_j) + \wa(z_k).
\]
To show that $m_{\{ijk\}}$ is the initial term of $[ijk]$ w.r.t. $\prec_{\wa}$, we consider two cases:

\textbf{Case 1.} Let $|\{i,j,k\}\cap I_s| > 1$. 
In this case $m_{\{ijk\}} = x_iy_jz_k$. It is enough to show that $\wa(x_iy_jz_k) - \wa(x_{\sigma(i)}y_{\sigma(j)}z_{\sigma(k)}) > 0$ for $\sigma \neq {\rm id}$. If $\sigma(k) \neq k$, we know that $k > \sigma(k)$. Then 
\[
\wa(x_iy_jz_k) - \wa(x_{\sigma(i)}y_{\sigma(j)}z_{\sigma(k)}) \geq (n-2)(k - \sigma(k)) + \wa(y_j) - \wa(y_{\sigma(j)}).
\]
It follows by definition of $\prec_{\wa}$ that $ |\wa(y_j) - \wa(y_{\sigma(j)})|  < n-2 $  and $\wa(x_iy_jz_k) - \wa(x_{\sigma(i)}y_{\sigma(j)}z_{\sigma(k)}) > 0 $.
If $\sigma(k) = k$, then 
$\sigma(j) = i$ as $\sigma \neq {\rm id}$.
Hence, 
\[
\wa(x_iy_jz_k) - \wa(x_{\sigma(i)}y_{\sigma(j)}z_{\sigma(k)}) =  \wa(y_j) - \wa(y_{\sigma(j)}) = \wa(y_j) - \wa(y_i).
\]
As $|\{i,j,k\}\cap I_s| > 1$, we have
$\{i,j\}\subset I_s$ and so 
$\wa(y_j) - \wa(y_{i}) =  j - i > 0$.

\medskip
\textbf{Case 2.} Let $|\{i,j,k\}\cap I_s| = 1$. Then $m_{\{ijk\}} = x_jy_iz_k$. 
%
If $\sigma(k) \neq k$, then:
\[
\wa(x_jy_iz_k) - \wa(x_{\sigma(j)}y_{\sigma(i)}z_{\sigma(k)}) \geq (n-2)(k - \sigma(k)) + \wa(y_i) - \wa(y_{\sigma(i)}).
\]  
We know that $k > \sigma(k)$, hence 
$|\wa(y_i) - \wa(y_{\sigma(i)})|< n-2 $. (Here $|.|$ denote the absolute value.) Therefore, $\wa(x_jy_iz_k) - \wa(x_{\sigma(j)}y_{\sigma(i)}z_{\sigma(k)}) > 0 $.
If $\sigma(k) = k$, then $\sigma(i) = j$ as $\sigma \neq \rm id$. 
Hence,
\[\wa(x_jy_iz_k) - \wa(x_{\sigma(j)}y_{\sigma(i)}z_{\sigma(k)}) =  \wa(y_i) - \wa(y_{\sigma(i)}) = \wa(y_i) - \wa(y_j).
\]
As $|\{i,j,k\}\cap I_s| = 1$, we have that $i \in  I_s $ and $j \in I_t$ for some $t>s$ and it follows from the definition of $\prec_{\wa}$ that $\wa(y_i) - \wa(y_j) > 0$, which completes the proof. 

By \cite[Theorem 1.1]{conca2015universal}, 
the maximal minors of $X$ form a universal Gr\"obner basis for $I_X,$ 
hence $M_\a = \init_{\prec_{\wa}}(I_X)$. This also implies the second claim. 
\end{proof}

\subsection{Connection to toric degenerations of Grassmannians}
The Grassmannian $\Gr(3,n)$ is the space of $3$-dimensional linear subspaces of $\mathbb{C}^n$, which can be embedded into a projective spaces, using Pl\"ucker embedding as follows. 
On the level of rings, consider the polynomial ring $S$ on the Pl\"ucker variables $p_I$, corresponding to the maximal minors of $X$ on the columns indexed by $I=\{i,j,k\}$, and the polynomial ring $R$ on the variables $x_i,y_j,z_k$ of the matrix $X$. The Pl\"ucker embedding is obtained by the map: 
\[
  \phi: 
  S
  \to R \quad\text{such that}\quad
  p_I  \mapsto \det(X_I).
\]
A good candidate for a toric degeneration of Gr$(k,n)$ is given by deforming $\phi$ to a monomial map. This is done by sending each \emph{Pl\"ucker variable} $p_I$ to one of the summands of $\det(X_I)$.
In particular, the monomial Gr\"obner degeneration of the determinantal ideal $I_X$ provides such a degeneration.  
More precisely, given a matching field $\bf a$, we define the map $\phi$ such that $\phi(p_I)=m_I$. In \cite{KristinFatemeh}, it is established that the kernel of this map forms a binomial and prime ideal, offering a toric degeneration of Gr$(3,n)$. Furthermore, Theorem~\ref{thm:main} further demonstrates that the matching ideal itself can be derived as a weighted Gr\"obner degeneration, and varying the weights $w_{\a}$, we can move through Gr\"obner fan and generate further toric degenerations, as the kernel of the corresponding monomial map.

\smallskip
For general Grassmannians and flag varieties, there are prototypic examples of toric degenerations which are related to young tableaux, Gelfand-Tsetlin integrable systems, and their polytopes. In the case of the Grassmannian Gr$(2,n)$, there are many other toric degenerations generalizing this primary example. Namely, any trivalent tree with $n$ number of labeled leaves gives rise to a toric degeneration of Gr$(2,n)$. The toric variety is governed by the isomorphism type of the trivalent tree \cite{speyer2004tropical}. In particular, {\rm Trop~Gr}$(2,n)$ forms the space of phylogenetic trees. In particular, all such degenerations can be obtained from the matching field ideals; see e.g.~\cite{KristinFatemeh}.

\begin{example}\label{exam:tree}
Consider the Grassmannian Gr$(2,4)$, whose corresponding Pl\"ucker ideal 
is $I=\langle p_{12}p_{34}-p_{13}p_{24}+p_{14}p_{23}\rangle$. 
A toric degeneration
of Gr$(2,4)$ is given by the family
$I_t=\langle p_{12}p_{34}-p_{13}p_{24}+tp_{14}p_{23}\rangle$ for $t\in\mathbb{C}$.
Setting
$t = 1$ we obtain $I$, and setting $t = 0$ we get the toric ideal $I_0=\langle p_{12}p_{34}-p_{13}p_{24}\rangle$. The Gr\"obner fan of $I$ consists of seven cones (three 2-dimensional cones, three rays, and the origin). 
Moreover, {\rm Trop~Gr}$(2,n)$ is the subfan of this polyhedral fan on the three rays, all leading to toric varieties; see \cite{ardila2006bergman}.

\end{example}

\subsection{\bf Ordering the generators of $M_{\bf{a}}$} \label{block ordering}
Here, we further investigate the properties of the matching field ideals $M_{\a}$, in particular their minimal free resolutions and their Betti numbers. We begin by recalling the linear quotients property from \cite{HT,Cone}.  
A monomial ideal $M \subset R$ is said to have {\it linear quotients} if there exists an ordering of the generators $M= \langle m_1, \dots, m_k\rangle$ such that for each $j$ the colon ideal $\langle m_1,\ldots,m_{j-1}\rangle : m_j$ is generated by a subset of the variables $\{x_1, \dots, x_n\}$ called the set$(m_j)$, so that
$\langle m_1,\ldots,m_{j-1}\rangle : m_j = \langle x_{j_1}, \dots, x_{j_r}\rangle.$
In particular, for each generator $m_j$, with $j = 1, \dots, k$, we define 
${\rm set}(m_j) = \{k \in [n]: x_k \in \langle m_1,\dots,m_{j-1}\rangle:m_j\}. 
$

\smallskip
To simplify our notation, we use $(i,j,k)$ for the monomial $x_iy_jz_k$. (Note that in this expression we may have either $i<j<k$ or $j<i<k$).
For every pair of elements $(i,j,k)$ and $(\ell,u,v)$ in the ideal 
$M_{\bf{a}}$ with $j\in I_s$ and $u\in I_t$ we define the block ordering $\prec_{{\bf a}}$ such that  
\[
(i,j,k)\prec_{{\bf a}} (\ell,u,v)\quad\text{if and only if}\quad\begin{cases}
k>v, &\text{or}\\
k=v,\ s<t, &\text{or}\\
k=v,\ s=t,\ j>u, &\text{or}\\
k=v,\ s=t,\ j=u,\ i<\ell 
\end{cases}
\]

\begin{example}  We illustrate an ordering of the generators of $M_{{\bf{(4,3)}}}$. The elements are arranged in a series of successive tableaux from left to right. Within each tableau, the elements are ordered from left to right and from top to bottom.
\[\begin{tabular}{ cccc|ccc|ccc|ccc }
  $147\ 247\ 347\ 547\ 647$
\\
  $\quad\ \ 137\ 237\ 537\ 637$
   \\
    $\quad\ \ \quad\ \ 127\ 527\ 627$
   \\
    $\quad\ \ \quad\ \ \quad\ \ 517\ 617$
       \\
    $\quad\ \ \quad\ \ \quad\ \ \quad\ \ 567$   
    \end{tabular}
   \quad
   \quad
   \begin{tabular}{ cccc|ccc|ccc|ccc }
    $146\ 246\ 346\ 546$
\\
  $\quad\ \ 136\ 236\ 536$
   \\
    $\quad\ \ \quad\ \ 126\ 526$
   \\
    $\quad\ \ \quad\ \ \quad\ \ 516$
    \end{tabular}
     \quad
   \quad
   \begin{tabular}{ cccc|ccc|ccc|ccc }
    $145\ 245\ 345$
\\
  $\quad\ \ 135\ 235$
   \\
    $\quad\ \ \quad\ \ 125$
    \end{tabular}
   \quad
   \quad
   \begin{tabular}{ cccc|ccc|ccc|ccc }
    $134\ 234$
\\
  $\quad\ \ 124$
    \end{tabular}
    \quad
   \quad
   \begin{tabular}{ cccc|ccc|ccc|ccc }
    $123$
    \end{tabular}
\]

\end{example}

For each monomial $x_\ell y_uz_v$, we define $S_\a(\ell,u,v)$ as the set of all monomials in $M_{\mathbf{a}}$ that precede $(\ell,u,v)$ in the block ordering $\prec_{\mathbf{a}}$, differing from $(\ell,u,v)$ by exactly one coordinate.

\begin{lemma}\label{lem:support_Lambda}
Consider the ordering $\prec_{{\bf a}}$. 
Then for each $(\ell,u,v)\in M_{\bf{a}}$ we have:
\[
|S_\a(\ell,u,v)|=\begin{cases}
\alpha_s-u+\ell-1+n-v, &\text{if $\ell,u\in I_s$ for $s<r$}\\
\ell-2+n-v, &\text{if $\ell\in I_s$ and $u\in I_t$ for $t<s$}\\
u-2+n-v, &\text{if $\ell,u\in I_r$.}
\end{cases}
\]
 \end{lemma}
 \begin{proof}
First note that for all $(\ell,u,v)$, the elements $\{(\ell,u,v+1),\ldots,(\ell,u,n)\}$ come before $(\ell,u,v)$ which implies that the variables $z_{v+1},\ldots,z_n$ are in the $S_\a(\ell,u,v)$. 
Now, we list all elements whose difference with 
$(\ell,u,v)$ is either in the first entry or in the second entry. 

\medskip
\noindent
{\bf Case 1.} Let $\ell,u\in I_s$. Then 
the elements coming before $(\ell,u,v)$ 
are:
\[
\{(\alpha_{s-1}+1,u,v),\ldots,(\ell-1,u,v)\}\cup\{(\ell,u+1,v),\ldots,(\ell,\alpha_s,v)\}\cup\{(\ell,1,v),\ldots,(\ell,\alpha_{s-1},v)\}.
\]
In particular, we have:
$S_\a(\ell,u,v) = \{ x_{\alpha_{s-1}+1}, \dots, x_{\ell-1}\} \cup \{y_{u+1}, \dots, y_{\alpha_s}\} \cup \{y_1, \dots, y_{\alpha_{s-1}}\}\cup\{z_{v+1},\ldots,z_n\}.
$

\medskip
\noindent
{\bf Case 2.} Let $\ell\in I_s$, $u\in I_t$ and $t<s$. By our construction, the elements listed below are those that precede $(\ell,u,v)$ (with one exception):
\[
\{(\alpha_{s-1}+1,u,v),\ldots,(u-1,u,v)\}\cup\{(\alpha_s+1,u,v),\ldots,(\ell-1,u,v)\}
\]
\[\cup\{(\ell,u+1,v),\ldots,(\ell,\alpha_{s},v)\}\cup\{(\ell,1,v),\ldots,(\ell,\alpha_{s-1},v)\},
\]
and we have
$
S_\a(\ell,u,v) = \{x_{\alpha_{s-1}+1}, \dots, x_{u-1}\} \cup \{ x_{\alpha_s+1}, \dots, x_{\ell-1}\} \cup \{y_{u+1}, \dots, y_{\alpha_s}\} \\\cup \{y_1, \dots, y_{\alpha_{s-1}}\}\cup\{z_{v+1},\ldots,z_n\}.
$
\medskip

\noindent
{\bf Case 3.} Let $\ell,u\in I_r$. Then 
$$
S_\a(\ell,u,v) = \{ x_{\alpha_{r-1}+1}, \dots, x_{\ell-1}\} \cup  \{y_{\ell+1}, \dots, y_{u-1}\} \cup \{y_1, \dots, y_{\alpha_{r-1}}\}\cup\{z_{v+1},\ldots,z_n\}.
$$
\end{proof}

\begin{theorem}\label{lem:blocklinearquotients}
$M_{\bf a}$ has the linear quotient property with respect to 
$\prec_{\bf a}$.
\end{theorem}
\begin{proof}
It suffices to show that $
((i,j,k) : (\ell,u,v)) \subset \langle S_\a(\ell,u,v)\rangle$ for any $(i,j,k) \prec_{\bf a} (\ell, u,v)$.
Note that $((i,j,k) : (\ell,u,v)) = \langle f\rangle$ for some monomial $f$ dividing $x_iy_jz_k$, and we need to show that this monomial lies in the ideal $\langle S_\a(\ell,u,v)\rangle$.

\textbf{Case 1.}
Let $k \neq v$. First note that, by definition of $\prec_{\bf a}$,  we have $k>v$. Thus, by Lemma~\ref{lem:support_Lambda}, we have $(\ell,u,k) \in M_{\bf a}$ and $z_k \in S_\a(\ell,u,v)$. Hence, $f \in \langle z_k\rangle$.

\textbf{Case 2.}  Suppose $k = v$, but $j \neq u$.  If $f = x_i$ or $y_j$ we are done, so assume $f = x_iy_j$.  We will show in all cases either $x_i \in S_\a(\ell,u,v)$ or $y_j \in S_\a(\ell,u,v)$.

First suppose $\ell, u \in I_s$ (and so in particular $\ell < u$).  If $j \in I_s$ then $j > u$ and so $\ell < j \in I_s$ and $(\ell,j,n) \prec_{\bf a} (\ell,u,v)$, hence $y_j \in S_\a(\ell,u,v)$.  If $j \notin I_s$, then we have $j \in I_t$ for $t < s$ and so $j < \ell, m$.  In particular, $(\ell,j,n) \prec_{\bf a} (\ell,u,v)$ so again $y_j \in S_\a(\ell,u,v)$ as desired.
Now suppose $\ell \in I_s, u \in I_t$ for $s \neq t$ (so in particular $u< \ell$).  If $j \in I_t$ then $j > u$ and since $\ell \in I_s$ we have $\ell > j$ and so $(\ell,j,n) \prec_{\bf a} (i,j,k)$ and again $y_j \in S_\a(\ell,u,v)$.
If $j \notin I_t$ then $j < u < \ell$.  If $j \notin I_s$ then $(\ell, j, n) \prec_{\bf a} (i,j,k)$ and again $y_j \in S_\a(\ell,u,v)$.  Otherwise, if $j, i \in I_s$ then $i < j$ and $(i,u,v) \prec_{\bf a} (i,j,k)$ and $x_j \in S_\a(\ell,u,v)$.  Finally, if $j \in I_s$ but $i \notin I_s$ then $i > j$.  Since $j,\ell \in I_s$ we also have $i > \ell > m$ and $(i,u,v) \prec_{\bf a} (\ell,u,v)$, hence $x_i \in S_\a(\ell,u,v)$.

\textbf{Case 3.} Suppose $k=v$ and $j = u$.  Then $f = x_i$, which completes the proof.
\end{proof}

It is shown in \cite{HT} that for any monomial ideal $M=\langle m_1,\ldots,m_k\rangle$ with linear quotient property, 
the Betti numbers can be explicitly described as follows.
\[\beta_\ell(M)=\sum_{j=1}^k{|{\rm set}(m_j)|\choose \ell}.\]

\begin{example}\label{exam:diag}
Consider the monomial ideal $M_{{\rm Diag}(n)}=\langle x_iy_jz_k:\ 1\leq i<j<k\leq n\rangle$. Note that each monomial is representing the diagonal term of the corresponding minor in a $3\times n$ matrix of indeterminants.
Consider a lexicographic ordering on the generators of $M_{{\rm Diag}(n)}$. Then 
\[
{\rm set}(x_iy_jz_k)=\{x_1,\ldots,x_{i-1}\}\cup\{y_{i+1},\ldots,y_{j-1}\}\cup\{z_{j+1},\ldots,z_{k-1}\}
\]
and hence, $|{\rm set}(x_iy_jz_k)|=(i-1)+(j-1-i)+(k-1-j)=k-3$.
 \end{example}

 As an immediate corollary of Theorem~\ref{lem:blocklinearquotients}, and Example~\ref{exam:diag} we have that:
\begin{corollary}\label{thm:betti}
The Betti numbers of the ideal $M_{{\rm Diag}(n)}$ are equal to
\[
\beta_\ell(M_{{\rm Diag}(n)})=\sum_{k=3}^{n}{k-1\choose 2}{k-3\choose \ell}.
\]
\end{corollary}

\subsection{Cellular resolution of $M_{\bf a}$}
Consider the matching field ideal $M_{\bf{a}}$, generated minimally by the monomials $m_I$ for all $3$-subsets $I$ of $[n]$.  We construct a CW-complex supporting the minimal free resolution of $M_{\bf{a}}$.
Let $H_{\bf a}$ be the $3$-hypergraph on the vertex set $V(H_{\bf a})=\{x_1,\ldots,x_n,y_1,\ldots,y_n,z_1,\ldots,z_n\}$, where the edges corresponding to the monomial generators of $M_{\bf a}$. 
Given a $3$-hypergraph  $H_{\bf a}$ we define a family of $2$ and $1$-hypergraphs as follows:
\begin{center}
    $z_k$-layer$(H_{\bf a})=\{x_iy_j|x_iy_jz_k \in E(H_{\bf a})\}$,\quad
$z_k y_l$-layer$(H_{\bf a})=\{x_i|x_iy_lz_k \in E(H_{\bf a})\}$.
\end{center}

\medskip
From the combinatorial description of the minimal generators of $\Ma$ it is easy to see that: 
\begin{lemma}\label{lem:1-layer}
Assume that $x_iy_jz_k \prec_{w_{\bf a}} x_ly_mz_n $. Then
 $z_n$-layer$(H_{\bf a})\subset z_k$-layer$(H_{\bf a})$.
Moreover, if $k = n$ then $z_ky_m$-layer$(H_{\bf a})\subset z_ky_j$-layer$(H_{\bf a})$.
\end{lemma}

We proceed by constructing another $3$-hypergraph $G_\a$, which is isomorphic to $H_{\bf a}$ and serves as a co-interval graph, as introduced in \cite{dochtermann2012cellular}. 
First, we order every edge of $H_{\bf a}$ according to the block ordering $\prec_{{\bf a}}$. Let us assume without loss of generality that we have $m$ distinct blocks.
In $z_m$-layer$(H_{\bf a})$ assume that we have $k$ distinct blocks with respect to the same ordering. Assume without loss of generality that $y_{m_1}$ is the lowest term in this layer and $z_my_{m_1}$- layer$(H_{\bf a})$ has $l$ distinct entries mainly $x_{k_1},x_{k_2},\ldots,x_{k_l}$.

\begin{definition}\label{def:G_a}
Given the $3$-hypergraph $H_\a$ we define the graph $G_{\bf a}$ on the vertex set
$V(G_{\bf a})=[m+k+l]$ with the edge set
$E(G_{\a}) = \{f(z_i)f(y_j)f(x_k)|x_ky_jz_i \in E(H_{\bf a})\}$, where $f$ is the map from the set $S_{\bf a}=\{z_1,z_2,\ldots,z_m,y_{m_1},y_{m_2},\ldots,y_{m_k},\\ x_{k_1},x_{k_2},\ldots,x_{k_l}\}$ to $[m+k+l]$ as follows:
\begin{itemize}
    \item $f(z_i)=m-i+1$ for $1\leq i\leq m$
    \item $f(y_{m_i})=m+i$ for $1\leq i\leq k$
    \item $f(x_{k_i})=m+k+i$ for $1\leq i\leq l$.
\end{itemize}
\end{definition}

We recall the notion of the $v$-layer and co-interval graph from \cite{dochtermann2012cellular}.
Let $H$ be a $d$-graph and let $v \in V(H)$ be some vertex. Then the $v$–layer
of $H$ is a $(d - 1)$-graph on $V \backslash v$ with edge set 
\hspace{1cm}

$
\quad\quad\quad\{v_1v_2\dots v_{d-1} | vv_1v_2\dots v_{d-1} \in E(H) \text{ and } v < v_1, v_2,\dots , v_{d-1}\}.
$
\begin{definition}
The class of co-interval $d$-graphs is defined recursively as follows.
Any $1$-graph is co-interval. For $d > 1$, the finite $d$-graph $H$ with vertex set $V(H) \subseteq Z$ is
co-interval if
\begin{enumerate}
    \item for every $i \in V (H)$ the $i$-layer of H is co-interval.
    \item for every pair $i < j$ of vertices, the $j$-layer of $H$ is a subgraph of its $i$-layer.
\end{enumerate}
\end{definition}
\begin{theorem}
The $3$-hypergraph $G_\a$ is a co-interval graph.
\end{theorem}
\begin{proof}
It is enough to prove the following two claims:
\begin{enumerate}
    \item For every $i \in V(G_\a)$ the $i$-layer of $G_{\a}$ is co-interval.
    \item For every pair $i<j$ of vertices, the $j$-layer of $G_{\bf a}$ is a subgraph of its $i$-layer.
\end{enumerate}
The construction of $G_\a$ 
implies that $j$-layer$(G_\a)\subseteq i$-layer$(G_\a)$ for $i<j$.
More precisely, $j$-layer$(G_\a)$ is the induced subgraph of the $i$-layer$(G_\a)$ since in our original hypergraph $H_\a$ the corresponding $z_{f^{-1}(j)}$-layer$(H_\a)$ is an induced subgraph of $z_{f^{-1}(i)}$-layer$(H_\a)$. Hence, it is sufficient to show that $1$-layer$(G_\a)$ is co-interval since any induced subgraph of a co-interval graph is also co-interval (see \cite[Proposition 4.2]{dochtermann2012cellular}).
Therefore, we only need to ensure that $j$-layer$(\text{1-layer}(G_\a))\subseteq i$-layer$(\text{1-layer}(G_\a))$ for all $i<j$ which is directly implied by Lemma~\ref{lem:1-layer}.
\end{proof}

\begin{corollary}
The polyhedral complex $X_{G_\a}$ supports a minimal cellular resolution of the ideal~$M_\a$. 
\begin{proof}
The map $f$ from Definition~\ref{def:G_a} induces a natural ring isomorphism: 
$$
k[z_1,\ldots,z_m,y_{m_1},\ldots,y_{m_k},x_{k_1},,\ldots,x_{k_l}] \rightarrow k[t_1,\ldots,t_{m+k+l}].
$$ 
We denote $N_\a$ for the ${\rm image}(f)$. The complex $X_{G_\a}$ supports a minimal cellular resolution of the monomial ideal $N_\a$ and $N_\a \cong M_\a$ (see \cite[Theorem 4.4]{dochtermann2012cellular}). 
\end{proof}
\end{corollary}

\begin{remark}
    Constructing a CW complex whose faces encode the free resolutions of the ideal is introduced in \cite{BS98} by Bayer and Sturmfels, and since then, it has been established that several classes of monomial ideals arising from graphs, matroids, and posets have cellular resolutions. See, e.g.~\cite{novik2002syzygies,cover,FarbodFatemeh,mohammadi2024family,clarke2024minimal,divisor}.
\end{remark}

\begin{example}
Let $n=5$ and  $\a=[123|45]$. Then the edges of the graph
$H_\a$ are:
123, 134, 234, 124, 135, 235, 435, 125, 425, 415.
Applying the map $f$ from Definition~\ref{def:G_a}, 
the edges of the graph $G_\a$ are:
$
357, 247, 248, 257, 147, 148, 149, 157, 159, 169.
$
The polyhedral complex supporting the minimal free resolution of $M_\a$ is shown 
in Figure~\ref{fig:example_cell}. The Betti numbers of $M_\a$ are: $\beta_0=10,\ \beta_1=15,\ \beta_2=6$. 

\begin{figure}[h]
    \centering
    \includegraphics[width=70mm]{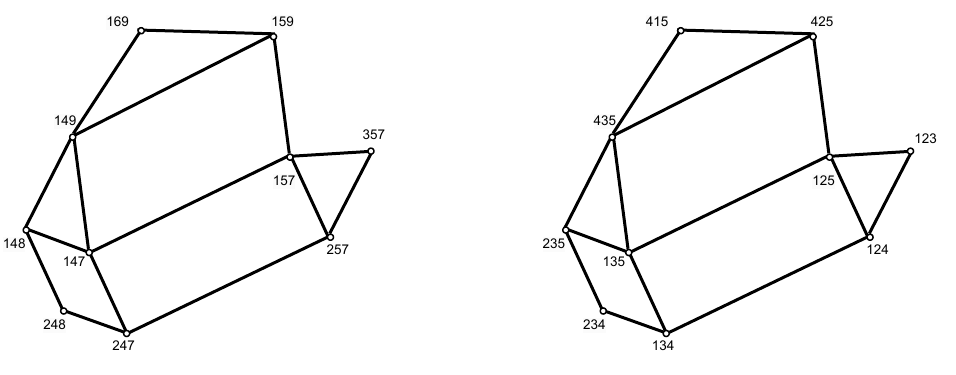}
   \caption{Polyhedral complex $X_{G_\a}$ with labelling of $M_\a$ (right) and labelling of $H_\a$ (left).}
    \label{fig:example_cell}
\end{figure}
\end{example}

\begin{credits}
\subsubsection{\ackname}
The author would like to thank Oliver Clarke and Janet Page for helpful discussions. The work was partially supported by the grants G0F5921N (Odysseus programme) and G023721N from the Research Foundation - Flanders (FWO), the UiT Aurora project MASCOT and the grant iBOF/23/064 from the KU Leuven.

\subsubsection{\discintname}
The authors have no competing interests to declare that are relevant to the content of this article. 
\end{credits}
%
%
%
%

\end{document}